\documentclass[11pt,a4paper]{amsart}

\title[A note on smooth forms on analytic spaces]{A note on smooth forms on analytic spaces}

\author{Mats Andersson \& H{\aa}kan Samuelsson Kalm}
\address{M. Andersson, H. Samuelsson Kalm, Department of Mathematical Sciences, Division of Algebra and Geometry, University of Gothenburg and 
Chalmers University of Technology, SE-412 96 G\"{o}teborg, Sweden}
\email{matsa@chalmers.se, hasam@chalmers.se}

\date{\today}

\usepackage[english]{babel}
\usepackage{amsmath,calligra}
\usepackage{amsthm,amssymb,latexsym}
\usepackage[all,cmtip]{xy}
\usepackage{url,ae}
\usepackage{t1enc}
\usepackage{bm}
\usepackage{mathrsfs}
\usepackage{graphicx}
\usepackage{geometry}
\newtheorem{proposition}{Proposition}[section]
\newtheorem{theorem}[proposition]{Theorem}
\newtheorem{lemma}[proposition]{Lemma}
\newtheorem{corollary}[proposition]{Corollary}

\theoremstyle{definition}

\newtheorem{example}[proposition]{Example}
\newtheorem{remark}[proposition]{Remark}

\numberwithin{equation}{section}

\DeclareMathOperator{\Hom}{\mathscr{H}\text{\kern -3pt {\calligra\Large om}}\,}
\DeclareMathOperator{\Ext}{\mathscr{E}\text{\kern -3pt {\calligra\Large xt}}\,\,}
\DeclareMathOperator{\Image}{\mathscr{I}\text{\kern -3pt {\calligra\Large m}}\,}
\DeclareMathOperator{\Ker}{\mathscr{K}\text{\kern -3pt {\calligra\Large er}}\,}
\newcommand{\C}{\mathbb{C}}
\newcommand{\debar}{\bar{\partial}}

\newcommand{\PM}{\mathscr{P} \kern -3pt \mathscr{M}}

\newcommand{\CH}{\mathscr{C} \kern -2pt \mathscr{H}}

\def\newop#1{\expandafter\def\csname #1\endcsname{\mathop{\rm #1}\nolimits}}
\newop{span}

\begin{document}
\nocite{*}
\bibliographystyle{plain}

\begin{abstract}
We prove that any smooth mapping between reduced analytic spaces induces a natural
pullback operation on smooth differential forms.
\end{abstract}

\maketitle
\thispagestyle{empty}

\section{Introduction}
There is a natural notion of smooth differential forms on any reduced analytic space.
The dual objects are the currents. Such forms and currents 
have turned out to be useful tools, e.g., in \cite{Barlet-book, Barlet-alpha, HL},
in the analytic approach to
intersection theory \cite{ASWY, AESWY1}, and in the context of the $\debar$-equation on analytic spaces \cite{AS, RSW}.

It is desirable to be able to take the direct image of a current under a proper map $f\colon X\to Z$ between reduced analytic spaces.
By duality this amounts to take pullbacks of smooth forms. In some works, e.g., \cite{ASWY, AESWY1}, it is implicitly assumed 
that this is possible.
There is an obvious tentative definition of $f^*\phi$ for a smooth form $\phi$ on $Z$. It is however not clear that 
it gives a well-defined pullback operation, not even if $f$ and $\phi$ are holomorphic and $\phi$ has positive degree; this case is settled in  
\cite[Corollary~1.0.2]{Barlet-alpha}. 
The main problem is when $f$ is the inclusion of an analytic subvariety contained in $Z_{sing}$.
It was proved in \cite[III Corollary 2.4.11]{Barlet-book} that if $f$ is holomorphic, then the suggested definition indeed gives
 a functorial operation on smooth forms.
In this  short note we give a new proof of this fact. Moreover, we extend it to the case when $f$ is merely smooth, see Theorem~\ref{snabel} below.
Our result is implicitly claimed in \cite{BH}, see Remark~\ref{BHaltdef} below.   

\smallskip

{\bf Acknowledgment:} We are grateful to the referee for careful reading and important comments.

\section{Results}
Let $X$ be a reduced analytic space. Recall that, by definition, there is a neighborhood $U$ of any point in $X$ and an 
embedding\footnote{By an embedding we always mean a closed holomorphic embedding so that the image is an analytic subvariety.}
$i\colon U\rightarrow D$ in an open set $D\subset\C^N$ such that $U$ can be identified with its image. 
For notational convenience we will suppress $U$ and say that $i$ is a 
local embedding of $X$. A smooth $(p,q)$-form $\phi$ on $X_{reg}$ is smooth on $X$, $\phi\in\mathcal{E}^{p,q}(X)$,
if there is a smooth form $\varphi$ in $D$ such that 
\begin{equation*}
i|_{X_{reg}}^*\varphi=\phi.
\end{equation*} 
If $j\colon X\to D'$ is another local embedding, then the identity on $X$ induces a biholomorphism $i(X)\xrightarrow{\sim} j(X)$. Thus, 
again by definition, locally in $D$ and $D'$, there are holomorphic maps $g\colon D\to D'$
and $h\colon D'\to D$ such that $i=h\circ j$ and $j=g\circ i$. Since $h^*\varphi$ is smooth in $D'$ and  
\begin{equation*}
j|_{X_{reg}}^*h^*\varphi=\phi,
\end{equation*} 
it follows that the notion of smooth forms on $X$ is independent of embedding.

We will write $i^*\varphi$
for the image of $\varphi\in\mathcal{E}(D)$ in $\mathcal{E}(X)$.
Let $[i(X)]$ be the Lelong current of integration over $i(X)_{reg}$.
The kernel of $i^*$ is closed since 
\begin{equation*}
i^*\varphi=0 \quad \iff \quad \varphi\wedge [i(X)]=0.
\end{equation*}  
Thus, with the quotient topology $\mathcal{E}(X)=\mathcal{E}(D)/\text{Ker}\, i^*$ is a  Fr\'echet space.
To see that this topology is independent of the embedding, notice that it is defined by the semi-norms
\begin{equation*}
|\phi|_{X,i}:=\inf \{|\varphi|_D;\, i^*\varphi=\phi\},
\end{equation*}
where $|\cdot |_D$ are the semi-norms defining the topology on $\mathcal{E}(D)$.
Let $j$, $g$, and $h$ be as above and let $|\cdot |_{X,j}$ be the analogously defined semi-norms induced by $j$.
Since $j^*\psi=\phi$ implies that $i^*g^*\psi=\phi$ and since $g^*\colon \mathcal{E}(D')\to \mathcal{E}(D)$ 
is continuous 
we get
\begin{equation*}
|\phi|_{X,i}=\inf_{i^*\varphi=\phi} |\varphi|_D \leq  \inf_{j^*\psi=\phi} |g^*\psi|_{D}
\leq C\inf_{j^*\psi=\phi} |\psi|_{D'} = C|\phi|_{X,j}.
\end{equation*}
In the same way, $|\phi|_{X,j}\leq C'|\phi|_{X,i}$ and it follows that the semi-norms $|\phi|_{X,j}$ give the same topology
as $|\phi|_{X,i}$.

\smallskip

Let $\mathcal{E}_X^{p,q}$ be the sheaf of smooth
$(p,q)$-forms on $X$ and let $\mathcal{E}_X^r=\oplus_{p+q=r}\mathcal{E}_X^{p,q}$.
We say that a continuous map $f\colon X\to Z$ between reduced
analytic spaces is \emph{smooth} if $f^*\phi\in\mathcal{E}_X^{0}$ for any $\phi\in\mathcal{E}_Z^{0}$.
Notice that if $i\colon X\to D_X$ and $j\colon Z\to D_Z$ are local embeddings, then $f$ is the restriction to $i(X)$ of a smooth map
$D_X\to D_Z$.

\begin{theorem}\label{snabel}
Let $f\colon X\to Z$ be a smooth map between reduced analytic spaces. There is a well-defined 
map $f^*\colon \mathcal{E}^r(Z)\to\mathcal{E}^r(X)$ with the following property: If 
$\phi$ is a smooth form on $Z$, $i\colon X\rightarrow D_X$ and $\iota\colon Z\rightarrow D_Z$ are local embeddings,
$\varphi$ is a smooth form in $D_Z$ such that $\iota^*\varphi=\phi$, and $\tilde f\colon D_X\to D_Z$ is a smooth map
such that $\tilde f|_{i(X)}=f$,
then 
\begin{equation}\label{property}
f^*\phi = i^*\tilde f^*\varphi.
\end{equation} 
\end{theorem}

Assume that $g\colon X\to Y$ and $h\colon Y\to Z$ are smooth maps such that
$f=h\circ g$. Let $j\colon Y\rightarrow D_Y$
be a local embedding 
and 
$\tilde g\colon D_X\to D_Y$ and $\tilde h\colon D_Y\to D_Z$ smooth maps such that $\tilde g|_{i(X)}=g$ and $\tilde h|_{j(Y)}=h$, respectively. 
Notice that the restriction of $\tilde h\circ \tilde g$ to $i(X)$ is $f$.
If $\phi\in\mathcal{E}(Z)$ and $\phi=\iota^*\varphi$ it follows  by Theorem~\ref{snabel} that
$h^*\phi=j^*\tilde h^*\varphi$ and $g^*h^*\phi=i^*\tilde g^* \tilde h^*\varphi=f^*\phi$.
Hence, 
\begin{equation}\label{elsa}
f^*\phi=g^*h^*\phi, \quad \phi\in\mathcal{E}(Z).
\end{equation}

\begin{remark}\label{BHaltdef}
An a priori different definition of smooth forms on $X$ is 
given in \cite[Section~3.3]{BH}.
If $i\colon X\to D$ is a local embedding, then the space of smooth forms on $X$ is defined in  \cite[Section~3.3]{BH}
as $\mathcal{E}(D)/\mathcal{N}(D)$, where $\mathcal{N}(D)$ is the space of smooth forms $\varphi$ in $D$ such that 
for any smooth manifold $W$ and any smooth map $g\colon W\to D$ with $g(W)\subset X$ one has $g^*\varphi=0$. 

It is clear that $\mathcal{N}(D)\subset\text{Ker}\, i^*$
and it is in fact claimed in \cite[Section~3.3]{BH} that  
$\mathcal{N}(D)=\text{Ker}\, i^*$,  but we have not been able to find a proof in the literature.
It follows from Theorem~\ref{snabel} that the claim indeed is true: If $g\colon W\to D$ is a smooth map
with $g(W)\subset X$, then $g=i\circ \gamma$ for a smooth map $\gamma\colon W\to X$. In view of \eqref{elsa} thus
$g^*\varphi=\gamma^*i^*\varphi=0$ if $i^*\varphi=0$.
\end{remark}

\smallskip

The space of currents on $X$, $\mathscr{C}(X)$, is the dual of the space of test forms, i.e., compactly supported forms in $\mathcal{E}(X)$,
cf.\ \cite[Section~4.2]{HL}. Let $f\colon X\to Z$ be as in Theorem~\ref{snabel} and assume that $f$ is proper.
Then $f^*\phi$ is a test form on $X$ if $\phi$ is a test form on $Z$.
If $\mu$ is a current on $X$ thus $f_*\mu$ is a current on $Z$ defined by
\begin{equation}\label{snyting}
f_*\mu . \phi = \mu . f^*\phi.
\end{equation}

By Theorem~\ref{snabel} and \eqref{elsa} we get

\begin{corollary}
Let $f\colon X\to Z$ be a smooth proper map between reduced analytic spaces. Then the induced mapping
$f_*\colon\mathscr{C}(X)\to\mathscr{C}(Z)$
has the property that if $f=h\circ g$, where $g\colon X\to Y$ and $h\colon Y\to Z$ are smooth proper maps, then
\begin{equation*}
f_*\mu = h_*g_*\mu, \quad \mu\in\mathscr{C}(X). 
\end{equation*} 
\end{corollary}

\begin{example}
Suppose that $i\colon X\to D$ is an embedding 
and consider the induced mapping $i_*\colon\mathscr{C}(X)\to\mathscr{C}(D)$.
It follows from \eqref{snyting} and the definition of 
test forms on $X$ that $i_*$ is injective. Thus $\mathscr{C}(X)$ can be identified with its image $i_*\mathscr{C}(X)$.
In view of the definition of $\mathscr{C}(X)$ and \eqref{snyting} it follows that $i_*\mathscr{C}(X)$ is the set of currents
$\mu$ in $D$ such that $\mu . \varphi=0$ if $i^*\varphi=0$. 
Notice in particular that $i_* 1=[i(X)]$.
\end{example}

\section{Proofs}
We will prove Theorem~\ref{snabel} by showing 
that the right-hand side of \eqref{property} is independent of the choices of embeddings $i$, $\iota$ and extensions $\tilde f$ and $\varphi$
of $f$ and $\phi$, respectively.
The technical part is contained in Proposition~\ref{kul}, cf.\ \cite[Proposition~III 2.4.10]{Barlet-book} and 
\cite[Proposition~1.0.1]{Barlet-alpha}. 
We begin with the following
lemma.

\begin{lemma}\label{snar}
Let $M$ be a reduced analytic space, $N$ a complex manifold, and 
$p\colon M\to N$ a proper holomorphic map. If $\text{dim}\, N=d\geq 1$ and 
$\text{rank}_x\, p <d$ for all $x\in M_{reg}$, then $p$ is not surjective. 
\end{lemma}

\begin{proof}
If $M$ is smooth it follows from the constant rank theorem that $p$ cannot be surjective. 
If $M$ is not smooth, let $\pi\colon\widetilde M\to M$ be a Hironaka resolution of singularities.
Then $\widetilde M$ is smooth and $\tilde p:=p\circ \pi$ is a proper holomorphic map with the same image as $p$.
Since $\pi$ is a biholomorphism outside the exceptional divisor $E=\pi^{-1}(M_{sing})$ we have
$\text{rank}_x\, \tilde p < d$ for all $x\in \widetilde M\setminus E$.
By semi-continuity of the rank it follows that
$\text{rank}_x\, \tilde p < d$ for all $x\in \widetilde M$. By the constant rank theorem thus $\tilde p$ cannot be surjective.

\end{proof}

\begin{proposition}\label{kul}
Let $D\subset\C^N$ be an open set and let $\varphi$ be a smooth form in $D$. 
\begin{itemize}
\item[(i)] Let $W\subset V$ be analytic subsets of $D$.
If the pullback of $\varphi$ to $V_{reg}$ vanishes, then the pullback of $\varphi$ to $W_{reg}$ vanishes.
\item[(ii)] Let $W$ be a smooth not necessarily complex submanifold of $D$, let $V\subset D$ be an analytic subset, and assume that $W\subset V$.
If the pullback of $\varphi$ to $V_{reg}$ vanishes, then the pullback of $\varphi$ to $W$ vanishes.
\end{itemize}
\end{proposition}

\begin{proof}[Proof of part (i)]
We may assume that $W$ is irreducible of dimension $d$.
We may also assume that $\varphi$ has positive degree since a smooth function vanishing on $V_{reg}$
must vanish on $W$ by continuity.
The case $d=0$ is then clear since the pullback of a form of positive degree to discrete points necessarily vanishes.
Let $\tilde\pi\colon V'\to V$ be a Hironaka resolution of singularities. 
Suppose that 
$W'\subset V'$ is analytic and such that $\tilde\pi(W')=W$. Let $\pi=\tilde\pi|_{W'}$ and
let $\phi$ be the pullback of $\varphi$ to $W_{reg}$. Since the pullback of $\varphi$ under 
$W'\hookrightarrow V'\to V\hookrightarrow D$ is $0$, it follows that $\pi^*\phi=0$.
We will find such $W'$ and $\pi$ such that $\pi^*\phi=0$ implies $\phi=0$.

To begin with we set $W'=\tilde\pi^{-1}(W)$. If $\tilde\pi(W'_{sing})=W$, replace $W'$ by $W'_{sing}$.
Possibly repeating this we may assume that $\tilde\pi(W'_{sing})\nsubseteq W$. Thus 
$\tilde\pi(W'_{sing})$ is a proper analytic subset of $W$. Set $\pi=\tilde\pi|_{W'}$
and notice that $\pi\colon W'\to W$ is proper and surjective.

Let
\begin{equation*}
M=W'\setminus \pi^{-1}(W_{sing}\cup \pi(W'_{sing})), \quad N=W_{reg}\setminus \pi(W'_{sing}),
\end{equation*}
and let $p=\pi|_M$. Since $M$ is smooth and $p\colon M\to N$ is proper and surjective it follows from 
the constant rank theorem that there is $x\in M$ such that 
$\text{rank}_x\,p=d$. Since $d$ is the optimal rank of $p$ this holds for $x$ in a non-empty Zariski-open subset of $M$.
Let $\widetilde M=\{x\in M; \text{rank}_x\,p\leq d-1\}$ be the complement of this set. Then $\text{rank}_x\,p|_{\widetilde M}\leq d-1$
for all $x\in\widetilde M_{reg}$. By Lemma~\ref{snar}, $p(\widetilde M)\nsubseteq N$ and thus $p(\widetilde M)$ is 
a proper analytic subset of $N$. 

Now, $N\setminus p(\widetilde M)$ is a dense open subset of $W_{reg}$ and so it suffices to show that $\phi=0$ there.
However, $M\setminus p^{-1}p(\widetilde M)$ is a (non-empty) open subset of $M$ and in this set $p$ has constant rank 
$=d=\text{dim}\, W$. Thus, $p$ is locally a simple projection and it follows that if $p^*\phi=0$, then $\phi=0$.  

\smallskip

\noindent \emph{Proof of part (ii).}
We use induction over $\text{dim}\, V$. The case $\text{dim}\, V=0$ is clear so suppose that $\text{dim}\, V>0$. 

Take a point $w\in W$. If $w\in V_{reg}$, then there is  a neighborhood $U\subset W$ of $w$  contained in $V_{reg}$. 
Then clearly the pullback of $\varphi$ to $U$ vanishes.
Assume now that $w\in V_{sing}$. If there is a neighborhood $U\subset W$ of $w$ contained in $V_{sing}$, 
then the pullback of $\varphi$ to $U$ vanishes in view of 
the induction hypothesis and part (i) of this proposition. 
If not, then there is a sequence of points $w_j\in W$ converging to $w$ such that $w_j\in V_{reg}$. Then there are neighborhoods 
$U_j\subset W$ of 
$w_j$ contained in $V_{reg}$. The pullback of $\varphi$ to $U_j$ vanish. Now, $\varphi(w)$ is a multilinear mapping on 
$T_wD$ depending continuously
on $w$. Since the pullback of $\varphi$ to $U_j$ vanish the restriction of $\varphi(w_j)$ to $T_{w_j}W$ vanish.
By continuity thus the restriction of $\varphi(w)$ to $T_wW$ vanishes.

Hence, for any $w\in W$, the restriction of $\varphi(w)$ to $T_wW$ vanishes; thus the pullback of $\varphi$ to $W$ vanishes.
\end{proof}

\begin{proof}[Proof of Theorem~\ref{snabel}]
Let $\phi\in\mathcal{E}(Z)$ and let $f^*\phi$ be the form on $X_{reg}$ defined by the right-hand side of \eqref{property}. 
Clearly $f^*\phi$ is smooth on $X$. As mentioned, we will show that it is independent of the choices of extensions $\tilde f$ and $\varphi$ as well
as of the local embeddings. 

First assume that $X$ is smooth. The set $X_1\subset X$ of points where $f$ has maximal rank is open.
By the constant rank theorem
each point $x\in X_1$ has a neighborhood
$U_x$ such that $f|_{U_x}$ is a submersion onto a smooth submanifold $f(U_x)$ of $D_Z$ contained in $Z$.
By  Proposition~\ref{kul} (ii), if $\varphi$ is a smooth form in $D_Z$ such that $\iota^*\varphi=\phi$, then
the pullback of $\varphi$ to $f(U_x)$ only depends on the pullback of $\varphi$ to $Z_{reg}$, i.e., only on $\phi$.
Thus, $f^*\phi$ is well-defined in $U_x$.
Hence, $f^*\phi$ is well-defined in $X_1$ and so, by continuity, well-defined in the closure $\overline X_1$.
Repeating the argument with $X$ and $f$ replaced by $X\setminus \overline X_1$ and 
$f|_{X\setminus \overline X_1}$ it follows that $f^*\phi$ is well-defined in $\overline X_2$,
where $X_2\subset X\setminus \overline X_1$ is the set of points where $f|_{X\setminus \overline X_1}$ has maximal rank.
Notice that this rank is strictly less than the rank of $f$ in $X_1$. Thus, continuing  the process of constructing such open sets $X_k$, 
after a finite number of steps we get $X_k=\emptyset$. Since $X=\cup_j \overline X_j$, $f^*\phi$ is well-defined in $X$.

In the case of a general $X$ we restrict $f$ to $X_{reg}$ and conclude that 
$f^*\phi$ is well-defined on $X_{reg}$, which by definition means that $f^*\phi$ is well-defined.

\end{proof}

\end{document}